\documentclass[11pt, oneside]{article}   	
\usepackage{geometry}                		
\usepackage{graphicx}				
\usepackage{amssymb}
\usepackage{amsmath}
\usepackage{graphicx}				
\usepackage{amssymb}
\usepackage{amsmath}
\usepackage{amsthm}
\usepackage{harpoon}
\usepackage{accents}

\usepackage{tikz}  
\usepackage{float}
\restylefloat{table}
\usepackage{mathrsfs}
\usepackage{verbatim}

\usetikzlibrary{positioning,chains,fit,shapes,calc}  

\newtheorem{theorem}{Theorem}
\newtheorem{lemma}{Lemma}
\newtheorem{definition}{Definition}

\newtheorem{remark}{Remark}

\title{On the inequalities in Hermite's theorem for a real polynomial to have real zeros}
\author{Mario DeFranco}

\begin{document}
\maketitle{}

\abstract{We prove expressions for the inequalities in Hermite's theorem which are conditions for a real polynomial to have real zeros. These expressions generalize the discriminant of a quadratic polynomial and the expression of J. Ma\u{r}\'ik for a cubic polynomial. We show that the $(k+1)$-th minor of the Hermite matrix associated a polynomial $p(x)$ is equal to the $k$-th minor of another matrix we call $E(n)$ times $n^{k-1}$ and a simple integer. To prove this equivalence, we prove generalizations of the discriminant of a polynomial and analyze certain labeled directed graphs. To define this matrix $E(n)$ we define functions $M(m_2,m_1,n)$ which are positive if the zeros of $p(x)$ are positive.}

\section{Introduction}\label{intro}

Let $p(x)$ be monic polynomial of degree $n$ with real coefficients
\[
p(x) = \sum_{i=0}^n a_i x^n.
\]
The Hermite theorem (see \cite{Obrechkoff}) describes how to determine the number of real zeros of $p(x)$ by checking certain inequalities involving the coefficients $a_i$. These inequalities are defined in terms of the leading principal minors of a certain matrix which we denote by $H(n)$. This paper proves that those minors are equal to the leading principal minors of another matrix $E(n)$ which we define in Section \ref{E}. The matrix $E(n)$ thus provides an alternative way of expressing these inequalities. 

We now state the results more specifically. Let $z_1, ..., z_n$ denote the zeros of monic polynomial $p(x)$ of degree $n$ with real coefficients. For integer $k \geq 0$, denote the power-sum functions by
\[
p_k(z_1, ..., z_n) = \sum_{i=1}^n z_i^k.
\]
The Hermite matrix $H(n)$ associated to $p(x)$ is the $n \times n$ matrix whose entries $H(n)_{i,j}$ are 
\[
H(n)_{i,j} = p_{i+j-2}(z_1,...,z_n).
\]
Let $\Delta_k(H(n))$ denote the determinant of the upper left $k \times k$ submatrix of $H(n)$. This determinant is known as the $k$-th leading principal minor with 
\[
\Delta_0(H(n))=1 \text { and }  \Delta_n(H(n))= \det(H(n)).
\]
Hermite's theorem then says that the zeros $z_1, ..., z_n$ are all real if and only if 
\begin{equation} \label{delta inequalities}
\Delta_{k}(H(n)) > 0
\end{equation}
for $1 \leq k \leq n$. Since power-sum functions are symmetric functions of the $z_i$, they may be expressed in terms of the elementary symmetric functions 
\[
e_k = e_k(z_1,...,z_n) = \sum_{1 \leq i_1< ...<i_k \leq n} \prod_{j=1}^k z_{i_j}.
\]
These elementary symmetric functions are then the coefficients 
\[
a_{n-k} = (-1)^k e_k(z_1,...,z_n).
\]
Thus the expressions $\Delta_{k}(H(n))$ may be expressed in terms of the $a_i$ by converting the power-sum functions into elementary symmetric functions, for example by applying the Newton-Girard identities.

We define an infinite matrix $E(n)$ and prove for $1\leq k \leq n-1$ 
\[
 (\prod_{i=1}^k i!)^2 n^{k-1}\Delta_{k+1} (H(n))=\Delta_{k}(E(n)).
\]
We explicitly express the coefficients of $E(n)$ in terms the $a_i$ without using the Newton-Girard identities.  To do this, we define the functions $M(m_1,m_2,n)$. We prove that these functions can be expressed as homogenous degree-2 polynomials evaluated at the $a_i$ and whose coefficients are linear functions of $n$. We prove that the entries of $E(n)$ are finite positive-integer linear combinations of the the $M(m_1,m_2,n)$. Thus the $M(m_1,m_2,n)$ reveal more structure to the inequalities \eqref{delta inequalities}. We also prove that each $M(m_1,m_2,n)$ is a polynomial with positive coefficients in the zeros $z_i$ of $p(x)$; therefore if the $z_i$ are all positive, then so are the $M(m_1,m_2,n)$. Thus if we fix $m_1$ and $m_2$ and let $n$ vary, the functions $M(m_1,m_2,n)$ provide a sequence of checks on the for the positivity of the zeros of $p(x)$. This sequence could be useful in establishing patterns for applying these inequalities to the Jensen polynomials for general entire functions. 
We present the upper left $3 \times 3$ submatrix of $E(n)$ in terms of the $M(m_1, m_2, n) = M_{m_1,m_2}$ functions and the elementary symmetric functions $e_k$:
\[
E(n) = \begin{bmatrix} 
M_{10}& M_{11}& M_{12}&\\ 
M_{11}&M_{20}+M_{12}&2M_{21}+M_{13}&... \\ 
M_{12}&2M_{21}+M_{13}& 2M_{30}+4M_{22}+M_{14}&\\ 
&  \vdots && \end{bmatrix}=
\]
\[
\begin{bmatrix} 
(n-1)e_1^2-2ne_0e_2 & (n-2)e_1 e_2 -3ne_0 e_3 &(2n-6)e_1e_3 -8ne_0 e_4\\ 
(n-2)e_1 e_2 -3ne_0&(2n-4)e_2^2 -2ne_1e_3-4ne_0e_4   &(4n-12)e_2e_3 -6ne_1e_4 - 10n e_0 e_5 \cdots \\ 
(2n-6)e_1e_3 -8ne_0 e_4&(4n-12)e_2e_3 -6ne_1e_4 - 10n e_0 e_5& (12n-36)e_3^2-8ne_2e_4-16ne_1e_5 - 24ne_0e_6\\ 
&  \vdots & \end{bmatrix}
\]
These expressions in $E(n)$ also generalize the discriminant of the quadratic equation and the expression of J. Ma\u{r}\'ik. That is, when $n=2$, 
\[
E(2)_{1,1} = a_1^2-4 a_0 a_2
\]
and when $n=3$ we get a $2 \times 2$ matrix whose determinant directly yields the expression of J.  Ma\u{r}\'ik:
\[
4(a_1^2-3a_0a_2)(a_2^2-3a_1a_3) - (a_1a_2-9a_0a_3)^2,
\]
see \cite{Marik} and \cite{Dimitrov}. Determinant expressions have been used by Csordas, Norfolk, and Varga in \cite{Csordas}
 and by Dimitrov and Lucas in \cite{Dimitrov} to prove that the Jensen polynomials of degree 2 and degree 3, respectively, for the Riemann xi function have real zeros. Thus the determinant expressions and insight into their structure could be useful for establishing the real zeros of these Jensen polynomials in general.
 
 We also note that minors of the Bezoutian matrix $B=B(p(x), p'(x))$ (see  \cite{Obrechkoff}) are another way to obtain inequalities for the real zeros of $p(x)$. This matrix entries are also degree-2 homogenous polynomials evaluated at $a_i$. It is different from $E(n)$ in that the entries do not depend on $n$ and uses $(k+1)$-st minors $\Delta_{k+1}(B)$ to get conditions corresponding to $\Delta_k(E(n))$. Thus $E(n)$ provides an alternative to the Bezoutian matrix and decrease the dimensions of the minors by 1. 
 
We describe the layout of this paper. In Section \ref{H formula} we prove a formula for $\Delta_{k+1}(H(n))$ which generalizes the discriminant of a polynomial using the Schur polynomials. We define $E(n)$ and $M(m_1,m_2,n)$ in Section \ref{E}. In characterizing the $M(m_1,m_2,n)$ functions we define ``incomplete" monomial and elementary symmetric functions; that is, symmetric functions whose arguments are a subset of the $z_i$. We then prove a formula for $\Delta_{k}(E(n))$ in Section \ref{E formula}. We show that these two formulas are equal up to a factor of $n^{k-1}$. To do this we analyze certain label directed graphs.  
 
\section{Formula for leading principal minors of the Hermite matrix $H(n)$}\label{H formula}
We fix a positive integer $n$ and use the indeterminates $z_1, ..., z_n$. These $z_i$ correspond to the $n$ roots of a polynomial of degree $n$. We will use the following standard symmetric functions in the $z_i$:
\begin{definition}
For integer $k \geq 0$, denote the power-sum functions by
\[
p_k = \sum_{i=1}^n z_i^k.
\]
For integer $k \geq 1$, denote the elementary-symmetric functions by
\[
e_k = \sum_{1\leq i_1 < ...< i_k \leq n} \prod_{j=1}^k z_{i_j}.
\]
In this notation we leave the $n$ and $z$ dependence implicit.
\end{definition}

The Hermite matrix $H(n)$ is defined in terms of the power-sum symmetric functions $p_k$. In our formulas we will make use of a generalization of the Hermite matrix which we define next.
\begin{definition}

Let $\lambda =( \lambda_1,\lambda_2, ...,\lambda_k)$ be a $k$-tuple of integers with $0\leq \lambda_i <\lambda_{i+1}$. 
Define $H(\lambda; n)$ to be the $k\times k$ matrix with entries
\[
H(\lambda; n)_{i,j} = p_{\lambda_i+j-1}.
\]
For $\lambda = (0,1,2,..., k-1)$, denote the Hermite matrix $H(n)$ by 
\[
H(n) = H(\lambda;n).
\]
\end{definition}
We will prove formulas for the leading principal minors $\Delta_k$ of $H(\lambda;n)$. 
\begin{definition} 
Let $F$ be an infinite matrix with entries $F_{i,j}$. Let $F_k$ denote the $k \times k$ submatrix with entries $(F_k)_{i,j}$ for $1 \leq i,j \leq k$. The denote 
\[
\Delta_k(F)  = \det(F_k).
\] 
\end{definition} 

\begin{definition} 
Let $\lambda$ be a $k$-tuple of integers
\[
\lambda = (\lambda_1, ..., \lambda_k).
\]
Let $x_1, ..., x_k$ be $k$ indeterminates. 
Then let $V(x_1, x_2, ..., x_k; \lambda)$ denote the $k \times k$ matrix with entries
\[
V(x_1, x_2, ..., x_k; \lambda)_{i,j} = x_{j}^{\lambda_i}. 
\]
Also denote 
\[
D(x_1, ..., x_k) = \det(V(x_1, ..., x_k; (0,1,2,..., k-1))) = \prod_{1 \leq i < j \leq k} (x_j - x_i)
\]
with
\[
D(x_1) =1.
\]
Let $S(x_1, x_2, ..., x_k; \lambda)$ denote the Schur polyonomial 
\[
 S(x_1, x_2, ..., x_k; \lambda) = \frac{\det(V(x_1, x_2, ..., x_k; \lambda))}{\prod_{1\leq i < j \leq k} (x_j - x_i )}.
\]
\end{definition} 
\begin{definition}
Let $C(k,n)$ denote the set of all subsets of order $k$ of the set $\{ 1,2,...,n\}$.
 For $b \in C(k,n)$ 
 \[
 b = \{b_1, ..., b_k \}
 \]
 with $b_i < b_{i+1}$, 
  let $z(b)$ denote the $k$-tuple 
 \[
 z(b) = (z_{b_1}, ..., z_{b_k}).
 \]
 For $1 \leq i \leq k$, let $\hat{b}_i \in C(k-1,n)$ be
 \[
\hat{b}_i = (b_1, ..., b_{i-1}, b_{i+1}, ..., b_{k}).  
\]
\end{definition}

\begin{theorem}
Let $\lambda$ be a $k$-tuple of integers 
\[
\lambda = (\lambda_1, ..., \lambda_k)
\]
with $0\leq \lambda_i \leq \lambda_{i+1}$. 
Then
\[
\det(H(\lambda; n)) = \sum_{b \in C(k,n)} S(z(b); \lambda) D(z(b))^2
\]
\end{theorem}
\begin{proof}
We use induction on $k$. The statement is true for $k=1$ because 
\[
D(z_i) = 1, \,\,\,\,S(z_i; (\lambda_1) ) = z_i^{\lambda_1}, \,\,\,\,\text{ and } \,\,\,\,\,\det(H(\lambda; n))= p_{\lambda_1}.
\]
Assume it is true for some $k\geq 1$. Let $\lambda = (\lambda_1, ..., \lambda_{k+1})$. Then we calculate 
\[
\det(H(\lambda; n))
\] 
by expanding along the rightmost column of the matrix $H(\lambda; n)$. Let $\hat{\lambda}_i$ denote the $k$-tuple 
 \[
 \hat{\lambda}_i = (\lambda_1, ..., \lambda_{i-1}, \lambda_{i+1}, ..., \lambda_{k+1}).
 \] 
Then 
\begin{align*} 
\det(H(\lambda; n)) &= \sum_{i=1}^{k+1}(-1)^{k+i}\det(H(\hat{\lambda}_{i}; n)) p_{\lambda_i+k} \\ 
 & = \sum_{b \in C(k,n)} D(z(b))^2 \sum_{i=1}^{k+1} (-1)^{k+i}   \frac{\det(V(z(b); \hat{\lambda}_i))}{ D(z(b))} p_{\lambda_{i}+k}
\end{align*}
by the induction hypothesis. Using the definition of $p_{\lambda_{i}+k}$ we re-write the last line of the above equation as 
\begin{align*}
&=  \sum_{b \in C(k,n)} D(z(b))^2  \sum_{i=1}^{k+1} (-1)^{k+i}   \frac{\det(V(z(b); \hat{\lambda}_i))}{ D(z(b))} \sum_{j=1}^n z_j^{\lambda_i+k}\\ 
&=  \sum_{b \in C(k,n)} D(z(b))^2 \sum_{j=1}^n z_j^k \sum_{i=1}^{k+1} (-1)^{k+i} \frac{\det(V(z(b); \hat{\lambda}_i))}{ D(z(b))} z_j^{\lambda_i}\\ 
&=  \sum_{b \in C(k,n)} D(z(b))^2 \sum_{j=1}^n z_j^k  \frac{\det(V((z(b),z_j); \lambda))}{ D(z(b))} 
\end{align*}
where $(z(b),z_j)$ denotes the $(k+1)$-tuple 
\[
(z(b),z_j) = (z_{b_1}, z_{b_2}, ..., z_{b_k}, z_j). 
\]
Continuing, we apply the definition of the Schur polynomial and obtain
\begin{align*}
&= \sum_{b \in C(k,n)} D(z(b))^2 \sum_{j=1}^n z_j^k  \frac{S((z(b),z_j);\lambda) D(z(b),z_j) }{ D(z(b))}.\\ 
\end{align*}
To the above expression we apply
\[
\frac{D(z(b),z_j) }{ D(z(b))} =  \prod_{i=1}^k (z_j-z_{b_i}),
\]
which yields
\begin{align*}
&  \sum_{b \in C(k,n)} D(z(b))^2 \sum_{j=1}^n z_j^k  S((z(b),z_j);\lambda)  \prod_{i=1}^k (z_j-z_{b_i})\\  
&= \sum_{b' \in C(k+1,n)}  S(z(b');\lambda) \sum_{i=1}^{k+1}  z_{b_i'}^kD(z(\hat{b'}_i))^2\prod_{l=1, \neq i}^{k+1} (z_{b_i'}-z_{b_l'}).
\end{align*}

Therefore we must show that for $b\in C(k+1,n)$ 
\[
\sum_{i=1}^{k+1}  z_{b_i}^kD(z(\hat{b}_i))^2\prod_{l=1, \neq i}^{k+1} (z_{b_i}-z_{b_l}) = D(z(b))^2
\]
This follows from Lemma \ref{identity} and completes the proof.
\end{proof}

\begin{lemma}\label{identity}
For integer $k\geq1$ and indeterminates $x_1, ..., x_{k+1}$,
\[
\sum_{l=1}^{k+1} (-1)^{l-1} x_l^{k} \prod_{1\leq i < j \leq k; i,j \neq l } (x_i - x_j)  = \prod_{1\leq i < j \leq k+1 } (x_i - x_j).  
\]
\end{lemma}
\begin{proof}
Expand the product on the right side of the lemma statement into monomial terms, treating the $x_i$ as non-commuting variables.  Each such term $m$ is indexed by a set $E(m)$ of $\frac{k(k+1)}{2}$ ordered pairs: if in the factor 
\[
(x_i-x_j)
\]
the $x_i$ contributes to $m$, then let $(i,j) \in E(m)$; otherwise $(j,i) \in E(m)$. Thus each $m$ corresponds to a directed graph $G(m)$ whose vertex set is 
\[
V = \{ 1, 2, ..., k+1\}
\]
and whose edge set is $E(m)$. We say that $(i,j)$ is an outgoing edge from the vertex $i$ and an incoming edge to the vertex $j$.

We claim that for every such $G$, either there is some vertex in $G(m)$ with all outgoing edges or there is a 3-cycle in $G(m)$. A ``cycle" means a directed cycle and a 3-cycle is a cycle with exactly 3 edges. We use induction on $k$. This statement is true for $k=1,2$. Assume it is true for some $k \geq 2$. Note in $G(m)$ every vertex has exactly $k$ edges.

We first show that cycles exist in $G(m)$ if there is no vertex with all outgoing edges. Suppose there is no vertex in $G(m)$ with all outgoing edges. If there exists a vertex $v$ with all incoming edges, then by the induction hypothesis, the subgraph $G(m) \backslash v$ has either a vertex $v'$ with all outgoing edges or a 3-cycle. If there isa 3-cycle, then we are done. Therefore assume there is such a $v'$ with all outgoing edges in $G(m)\backslash v$. Then since $v$ has all incoming edges by assumption, we have the edge $(v',v)$ in $G(m)$. Thus $v'$ has all outgoing edges in $G(m)$, contradicting the assumption of no such vertex in $G(m)$. Therefore if $G(m)$ has no vertex with all outgoing edges, then it has no vertex with all incoming edges.


Thus we can assume that no vertex in $G(m)$ has all outgoing edges and that every vertex has at least one incoming and at least one outgoing edge. Therefore there exists some cycle $C$ in $G(m)$ with at least 3 edges. Let $C$ consist of the vertices $v_1, ..., v_n$ with edges $(v_i, v_{i+1})$ and $(v_n,v_1)$. Then one of the triples $\{v_1, v_i, v_{i+1}\}$ for $2\leq i\leq n-1$
must constitute a 3-cylce. For if none of these triples were a 3-cylce, then that means we would have to have the directed edges $(v_1, v_i)$ for $2\leq i\leq n-1$. But then we have the triple $\{v_1, v_{n-1},v_n\}$ which would then be a 3-cycle. This proves the induction step.

Now we can prove the lemma. Allowing the $x_i$ to commute, the term 
\[
 (-1)^{l-1} x_l^{k} \prod_{1\leq i < j \leq k; i,j \neq l } (x_i - x_j)
\]
is the sum of all monomials whose associated graphs have the vertex $l$ with all outgoing edges. For a monomial $m$ such that $G(m)$ does not have a vertex with all outgoing edges, then we know from above that $G(m)$ has a 3-cycle. Let $C(m)$ be the 3-cycle whose triple of vertices $\{i,j,k\}$ is the smallest of all 3-cycles in $G(m)$ in the lexicographic ordering. Let $m'$ denote the monomial whose graph $G(m')$ is the same as $G(m)$ except that the cycle $C(m)$ has its three edges reversed. Then the monomial $m'$ has opposite sign to that of $m$. This bijection $m\mapsto m'$ shows that all such monomial terms cancel. This proves the lemma. 

\end{proof}

\section{The matrix $E(n)$ } \label{E}
We define the matrix $E(n)$. To define the entries of $E(n)$, we first define functions $M(m_1,m_2;n)$.

\subsection{The functions $M( m_1,m_2,n)$} \label{M}

\begin{definition}
Let $P(k,n)$ denote the set of all $k$-tuples $b$ 
\[
b = (b_1, ..., b_k)
\]
such that $ b_i \neq b_j$ for $i \neq j$ and $b_i \in \{1,...,n\}$.
For a finite non-increasing sequence of non-negative integers $d=\{ d_1, d_2, ..., d_l\}$, define 
\[
\mathrm{monomial}(d,n) = \sum_{b \in P(l,n)} \prod_{i=1}^l z_{b_i}^{d_i}.
\]
Define
\[
\mathrm{monomial}_2(m_2, m_1,n) =\frac{\mathrm{monomial}(\{2,2,...,2,1,1,...,1 \},n)}{m_2! m_1!}
\]
where there are $m_2$ 2's and $m_1$ 1's.  
\end{definition}
 \begin{lemma}\label{ememk formula}
 For integers $m,k \geq 0$,
\[
e_me_{m+k} = \sum_{i=0}^m {k+2i \choose i} \mathrm{monomial}_2(m-i,k+2i;n).
\]
\end{lemma}
\begin{proof}
As a function of the $z_i$, 
\begin{equation} \label{em em+k}
e_me_{m+k}
\end{equation}
is a symmetric polynomial. Since each $z_i$ appears with exponent at most 1 in each elementary-symmetric function, any product of two elementary symmetric functions is then a linear combination of the functions $\mathrm{monomial}_2(m_1,m_2;n)$. Expanding the product \eqref{em em+k} into monomials, we get terms of the form 
\[
z_{j_1}^2...z_{j_{m-i}}^2 z_{l_1}...z_{k+2i}
\] 
where $0 \leq i \leq m$; to see this, one $z_{j_h}$ factor comes from $e_m$ and another $z_{j_h}$ factor comes from $e_{m+k}$. That is, the term from $e_{m}$ and the term from $e_{m+k}$ overlap in $m-i$ indeterminates. Then the remaining indeterminates coming from $e_{m}$ are distinct from the remaining ones coming from $e_{m+k}(z;n)$. The total number of these indeterminates that do not overlap is 
\[
i+ (m+k - (m-i)) = k+2i. 
\] 
and $i$ of these indeterminates come from $e_{m}$. Thus there are $\displaystyle {k+2i \choose i}$ ways to make such a product of two terms. This proves the lemma.
\end{proof}
 \begin{lemma}\label{mono2 as e}
 For integers $m,k \geq 0$, 
  \[
\mathrm{monomial}_2(m,k,n) = \sum_{i=0}^m (-1)^i \frac{k+2i}{i}{k+i-1\choose i-1} e_{m-i}e_{m+k+i}
\]
 \end{lemma}
\begin{proof} 
Solving for $\mathrm{monomial}_2(m,k)$ using the system of equations given by Lemma \ref{ememk formula} gives 
\[
\mathrm{monomial}_2(m,k,n)= \sum_{i=0}^m c_i(k) e_{m-i}e_{m+k+i}
\]
where 
\begin{align*}
c_0(k) &= 1\\
c_i(k) &=- \sum_{j=1}^i {k+2j \choose j} c_{i-j}(k+2j) \text{ for } i \geq 1.  
\end{align*}
Then for $i \geq 1$
\[
c_i(k) =(-1)^i \frac{k+2i}{i}{k+i-1\choose i-1}. 
\]
We prove this by induction on $i$. It is true for $i=1$. Assume it is true for some $i\geq 1$. Then
\begin{align*}
&(-1)^{i+1} \frac{k+2i+2}{i+1}{k+i\choose i}+\sum_{j=1}^{i+1} {k+2j \choose j} c_{i+1-j}(k+2j) \\ 
=& (k+2i+2)\sum_{j=0}^{i+1} (-1)^j\frac{(k+2i-1-j)_i}{j!(i+1-j)!}\\ 
=& (k+2i+2)\frac{1}{i!} (\frac{d}{dt})^i t^{k+i}(1-t)^{i+1} |_{t=1}\\ 
=&0.
\end{align*}
This proves the induction step and the lemma.
\end{proof}

Now we define the functions $M(m_2,m_1, n)$.
\begin{definition}

Let $d_1$ denote the sequence 
\[
d_1 = \{2,2,...,2,1,1,...,1, 0\}
\]
where there are $m_2$ 2's and $m_1$ 1's. Let $d_2$ denote the sequence 
\[
d_2 = \{2,2,...,2,1,1,...,1\}
\]
where there are $m_2-1$ 2's and $m_1+2$ 1's.
Define 
\[
M(m_2,m_1, n) = \mathrm{monomial}(d_1,n) -  \mathrm{monomial}(d_2,n).
\]
\end{definition}
\begin{lemma} 
\begin{align*}
M(m_2,m_1,n) =&  m_2! m_1!(n-m_1-m_2)\sum_{i=0}^{m_2}(-1)^i \frac{(m_1+2i)(m_1+i-1)!}{i!m_1!}e_{m_2-i}e_{m_2+m_1+i} \\ 
&- (m_2-1)! (m_1+2)! \sum_{i=0}^{m_2-1}(-1)^i \frac{(m_1+2+2i)(m_1+i+1)!}{i!(m_1+2)!}e_{m_2-1-i}e_{m_2+m_1+1+i}
\end{align*}
\end{lemma}
\begin{proof}
This follows from applying the definitions and Lemma \ref{mono2 as e}.  
\end{proof}
\begin{definition} For a finite non-increasing sequence of non-negative integers $d=\{ d_1, d_2, ..., d_l\}$, define 
\[
\mathrm{monomial}_{\mathrm{inc}}(d,i,j;n) = \sum_{b \in P(l,n); i,j \notin b} \prod_{h=1}^l z_{b_h}^{d_h}.
\]
Define incomplete elementary-symmetric functions 
\[
e_{\mathrm{inc}}(k; i;n) = \sum_{1 \leq l_1<...<l_k\leq n; l_h\neq i }   z_{l_1}...z_{l_k} = e(z;k;n)|_{z_i=0}
\]
and
\[
e_{\mathrm{inc}}(k; i,j;n) = \sum_{1 \leq l_1<...<l_k\leq n; l_h\neq i,j }   z_{l_1}...z_{l_k} = e(z;k;n)|_{z_i=0, z_j=0}
\]

That is, $\mathrm{monomial}_{\mathrm{inc}}(d,i,j;n)$ and $e_{\mathrm{inc}}(k; i,j;n)$ are equal to $\mathrm{monomial}(d,n)$ and $e_k$ respectively, but without any terms that involve non-zero powers of $z_i$ and $z_j$.  
\end{definition} 

\begin{lemma} \label{M mono inc}
Let $d$ denote the sequence 
\[
d = \{ 2,2,...,2,1,1,...,1\}
\]
where there are $m_2-1$ 2's and $m_1$ 1's. Then
\[
M(m_2,m_1, n)  = \sum_{1 \leq i < j \leq n} \mathrm{monomial}_{\mathrm{inc}}(d,i,j;n)(z_i - z_j)^2.
\]
\end{lemma}
\begin{proof}
Applying the definitions we obtain
\begin{align*}
M(m_2,m_1, n) =& \sum_{b \in P(m_1+m_2+1,n)}z_{b_1}^2...z_{b_{m_2}}^2 z_{b_{m_2+1}}...z_{b_{m_2+m_1}} z_{b_{m_2+m_1+1}}^0\\
& - \sum_{b \in P(m_1+m_2+1,n)}z_{b_1}^2...z_{b_{m_2-1}}^2z_{b_{m_2}} z_{b_{m_2+1}}...z_{b_{m_2+m_1}} z_{b_{m_2+m_1+1}}. 
\end{align*}
Now we partition the set $P(m_1+m_2+1,n)$ into pairs $\{ b,b'\}$ where for any $b \in P(m_1+m_2+1,n)$, we let $b'$ be obtained from $b$ by switching the elements 
\[
b_{m_2} \text{ and } b_{m_2+m_1+1}.
\]
Then we get 
\begin{align*}
M(m_2,m_1, n) = &\sum_{\{ b, b'\}} z_{b_1}^2...z_{b_{m_2-1}}^2z_{b_{m_2+1}}...z_{b_{m_2+m_1}} (z_{b_{m_2}}^2 +z_{b_{m_2+m_1+1}}^2 - 2z_{b_{m_2}}z_{b_{m_2+m_1+1}})\\
= &\sum_{\{ b, b'\}} z_{b_1}^2...z_{b_{m_2-1}}^2z_{b_{m_2+1}}...z_{b_{m_2+m_1}} (z_{b_{m_2}}-z_{b_{m_2+m_1+1}})^2 \\ 
= & \sum_{1 \leq i < j \leq n} \sum_{b \in P(m_1+m_2-1,n); i,j\notin b} z_{b_1}^2...z_{b_{m_2-1}}^2z_{b_{m_2+1}}...z_{b_{m_2+m_1}} (z_{i}-z_{j})^2\\ 
=&  \sum_{1 \leq i < j \leq n} \mathrm{monomial}_{\mathrm{inc}}(d,i,j;n)(z_i - z_j)^2.
\end{align*}
This completes the proof.
\end{proof}

\subsection{The definition of $E(n)$} \label{definition of E}

Now we define the matrix $E(n)$. 
\begin{definition}
Define the infinite matrix $E(n)$ with entries $E(n)_{r,s}$
\[
E(n)_{r,s} =  (r-1)!(s-1)! \sum_{1 \leq i <j \leq n} e_{\mathrm{inc}}(r-1,n; i,j)e_{\mathrm{inc}}(s-1,n; i,j)(z_i - z_j)^2.
\]
\end{definition}

\begin{lemma} \label{E as M}
For integers $m,k \geq 0$ and $E(n)_{m,m+k}$ defined above,
\[
E(n)_{m+1,m+k+1}=\sum_{i=0}^m {m \choose i} \frac{(m+k)!}{(i+k)!} M(m+1-i,k+2i,n) 
\]
\end{lemma}
\begin{proof}
By Lemma \ref{M mono inc}, it is sufficient to prove 
\[
\sum_{i=0}^m {m \choose i} \frac{(m+k)!}{(i+k)!} (m-i)!(k+2i)! \mathrm{monomial}_2(m-i,k+2i,n) = (m!)(m+k)!e_me_{m+k}.
\]
 By Lemma \ref{ememk formula}, the coefficient of $\mathrm{monomial}_2(m-i,k+2i,n)$ in $e_m e_{m+k}$ is $\displaystyle { k+2i\choose i} $.
Then 
\[
\frac{m!(m+k)!}{(m-i)!(k+2i)!}{ k+2i\choose i} =  {m \choose i} \frac{(m+k)!}{(i+k)!}.
\]
This completes the proof.
\end{proof}

\begin{theorem}

\[
E(n)_{m+1,m+k+1}(n) = m!k!(n-m-k)e_m e_{m+k}+ n \sum_{i=0}^{m-1} A_i e_i e_{2m+k-i} 
\]
for some numbers $A_i$. 
\end{theorem}
\begin{proof}
 Using Lemma \ref{mono2 as e} and the  definition of $M(m,k,n)$ we have 
 \begin{align*}
 &M(m,k,n) = m!k!(n-m-k)e_m e_{m+k}  \\ 
 &+ \sum_{i=0}^{m-1} (m!k!(n-m-k)(-1)^{m-i} \frac{k+2m-2i}{m-i} {k+m-i-1 \choose k} \\ 
 &- (m-1)!(k+2)!(-1)^{m-i-1}\frac{k+2m-2i}{m-i-1} {k+m-i \choose k+2}) e_{m-i}e_{2m+k-i}
 \end{align*}
 Re-indexing $ i \mapsto m-i$ and applying Lemma \ref{E as M}, we get that the coefficient of $e_{m-i}e_{2m+k-i}$  in $E(n)_{m+1,m+k+1}$ that is constant in $n$ is 
 \begin{align}\nonumber
&(k+2-2i+2m) \sum_{j=0}^m (-1)^{m-i-j+1}{m \choose j} \frac{(m+k)!}{(j+k)!} \\ \label{summand}
\times & ((-1-j-k-m)(m-j+1)! \frac{(m+k+j-i)!}{(m-j-i+1)!}+(m-j)!\frac{(m+k+j-i+1)!}{(m-j-i)!}).
 \end{align}
 We must show that the above sum is 0 for $0 \leq i \leq m-1$.
We simplify line \eqref{summand} to obtain 
\[
\frac{m!(m+k)!}{j!(j+k)!}i(i-2m-k-2)(m+k+j-i)_{2j+k-1}.
\]
 This shows that the sum is 0 for $i=0$. We this must prove that 
 \[
 \sum_{j=0}^m (-1)^j \frac{(m+k+j-i)_{2j+k-1}}{j!(j+k)!}
 \]
 has a factor of $(i-h)$ for $1 \leq h \leq m-1$.
This follows from the identity  
\[
\sum_{j=0}^m (-1)^j \frac{(m+k+j-i)_j (m-i+1)_j}{j!(j+k)!} =(-1)^{m} \frac{\left(\prod_{h=1}^m (i-h)\right) \left( \prod_{h=1}^m (i-(2m+k)+h) \right)}{m!(m+k)!}.
\] 
To prove this we use induction on $m$. It is trie for $m=0$. Assume it is true for $m\geq0$. Then the induction step follows from the identity 
\[
-1+ \frac{(m+k-i+1)(m-i+1)}{j(j+k)} = \frac{(m-j-i+1)(m+k+j-i+1)}{j(j+k)}.
\] 
\end{proof}

\section{Formula for leading principal minors of $E(n)$} \label{E formula}

\begin{definition} 
For integer $k \geq 0$ and indeterminates $x_1$ and $x_2$, define $f_{k}(x_1,x_2)$
\[
f_{k}(x_1,x_2) = \sum_{j=0}^k x_1^{k-j} x_2^j.
\]
For a pair of integers $b = (b_1,b_2)$, we also use the notation
\[
f_k(b) = f_k(z_{b_1}, z_{b_2}).
\]
\end{definition}

\begin{lemma} \label{e inc f}
\begin{align} \label{e inc i}
e_{\mathrm{inc}}( k;i;n)&= \sum_{h=0}^k (-1)^h z_i^he_{k-h}  \\ \label{e inc ij}
e_{\mathrm{inc}}( i,j;n)&= \sum_{h=0}^k (-1)^h f_j(z_i,z_j)e_{k-h}
\end{align}
\end{lemma}
\begin{proof}
We have 
\[
e_{\mathrm{inc}}( k;i;n) = e_k - z_i e_{\mathrm{inc}}( k-1;i;n).
\] 
This implies equation \eqref{e inc i}.
We also have 
\[
e_{\mathrm{inc}}( k;i,j;n)  = e_k - z_i e_{\mathrm{inc}}( k;i;n) - z_j e_{\mathrm{inc}}( k;j;n)+ z_i z_j e_{\mathrm{inc}}( k-2;i,j;n). 
\]
The above equation combined with \eqref{e inc i} implies \eqref{e inc ij}.
\end{proof}

We will use the following definitions in Theorem \ref{E minor formula}.

\begin{definition}
Let $\mathrm{Pairs}(k,n)$ denote the set of elements $\beta$ where each element $\beta$ is a multi-set of $k$ pairs of integers:
\begin{equation} \label{beta definition}
\beta= \{ \beta(1), \beta(2), ..., \beta(k) \}
\end{equation}
where 
\[
\beta(i) = \{ \beta(i,1), \beta(i,2)\}
\]
such that $ \beta(i,1) <\beta(i,2)$ and each $\beta(i,j) \in \{1,2,...,n \}$.
Note that $\beta$ is a set: even though we have used an ordering of the pairs $\beta(1), ..., \beta(k)$ in the notation of \eqref{beta definition}, another ordering would result in the same element $\beta$.

Let $|\beta|$ denote
\[
|\beta| = \#\text{ distinct numbers that appear as $\beta(i,1)$ or $\beta(i,2)$}.
\]
For $ \beta \in \mathrm{Pairs}(k,n)$, define 
\[
\mathcal{D}(\beta) = \prod_{i=1}^k (z_{\beta(i,1)}- z_{\beta(i,2)})
\]
\end{definition}

\begin{definition} 
For $\beta \in \mathrm{Pairs}(k,n)$, define the $k \times k$ matrix $R_1(\beta)$ with entries $R_1(\beta)_{u,v}$ by 
\[
R_1(\beta)_{u,v} = e_{\mathrm{inc}}( u;\beta(v,1),\beta(v,2);n). 
\] 

Define the $k \times k$ matrix $R_2(\beta)$ with entries $R_2(\beta)_{u,v}$ by 
\[
R_2(\beta)_{u,v} = f_u(z_{\beta(v,1)}, z_{\beta(v,2)}). 
\] 

Define the $k \times k$ matrix $R_3(\beta)$ with entries $R_3(\beta)_{u,v}$ by 
\[
R_3(\beta)_{u,v} = z_{\beta(v,1)}^u- z_{\beta(v,2)}^u. 
\] 
\end{definition}

\begin{theorem} \label{E minor formula}
\[
\Delta_k(E(n)) = (\prod_{i=1}^ki!)^2\sum_{\beta \in \mathrm{Pairs}(k,n)} \det(R_3(\beta))^2
\]
\end{theorem}
\begin{proof}
We write the definition of $E(n)_{u,v}$ as
\[
E(n)_{u,v} = \sum_{\beta \in B(1,n)}  e_{\mathrm{inc}}(u;\beta(1,1),\beta(1,2)) e_{\mathrm{inc}}(v;\beta(1,1),\beta(1,2)) (z_{\beta(1,1)} -z_{\beta(1,2)} )^2.
\]
We first prove that 
\begin{equation}\label{R1}
\Delta_k(E(n)) =  (\prod_{i=1}^ki!)^2\sum_{\beta \in \mathrm{Pairs}(k,n)} \mathcal{D}(\beta)^2  \det(R_1(\beta))^2.
\end{equation}
When calculating the determinant $\Delta_k(E(n))$, we write as a sum over $\mathrm{Pairs}(k,n)$: 
\begin{align*}
\Delta_k(E(n)) &= (\prod_{i=1}^ki!)^2\sum_{\beta \in \mathrm{Pairs}(k,n)} \mathcal{D}(\beta)^2 \left( \sum_{\tau \in S_k} \sum_{\sigma \in S_k}\mathrm{sgn}(\sigma) \prod_{u=1}^k  e_{\mathrm{inc}}(u, \beta(\tau(u)) e_{\mathrm{inc}}(\sigma(u), \beta(\tau(u))\right) \\ 
\end{align*}
Now
\begin{align*}
&\sum_{\tau \in S_k} \sum_{\sigma \in S_k} \mathrm{sgn}(\sigma) \prod_{u=1}^k e_{\mathrm{inc}}(u, \beta(\tau(u)) e_{\mathrm{inc}}(\sigma(u), \beta(\tau(u))\\
&=  \sum_{\tau \in S_k} \sum_{\sigma \in S_k} \mathrm{sgn}(\sigma)\prod_{u=1}^k  e_{\mathrm{inc}}(u, \beta(\tau(u)) e_{\mathrm{inc}}(u, \beta(\tau(\sigma^{-1}(u))) \\ 
& = \sum_{\sigma_1 \in S_k} \sum_{\sigma_2 \in S_k} \mathrm{sgn}(\sigma_1)\mathrm{sgn}(\sigma_2)\prod_{u=1}^k  e_{\mathrm{inc}}(u, \beta(\sigma_1(u)) e_{\mathrm{inc}}(u, \beta(\sigma_2(u))
\end{align*}
where 
\[
\sigma_1 = \tau, \sigma_2 = \tau \sigma^{-1}.
\]
Thus continuing we get 

\begin{align*}
& = \left( \sum_{\sigma' \in S_k}\mathrm{sgn}(\sigma')\prod_{u=1}^k  e_{\mathrm{inc}}(u, \beta(\sigma'(u)))\right)^2 \\
&= \det(R_1(\beta))^2.
\end{align*}
This proves equation \eqref{R1}. 

We next prove that 
\[
\det(R_1(\beta))^2 = \det(R_2(\beta))^2.
\]
As a sum over $S_k$, each of the $k!$ terms in $\det(R_1)$ is of the form 
\begin{equation} \label{R1 form}
\mathrm{sgn}(\sigma)\prod_{u=1}^k  e_{\mathrm{inc}}(u, \beta(\sigma(u)))
\end{equation}
for some $\sigma \in S_k$. We apply Lemma \ref{e inc f} to write \eqref{R1 form} as 
\[
\mathrm{sgn}(\sigma) \sum_{g} h(g)\prod_{u=1}^k f_{g_u}(\beta(\sigma(u))
\]
where the sum is over all $k$-tuples $g$ 
\[
g = (g_1, ..., g_k)
\]
such that $0 \leq g_i \leq i-1$; and $h(g)$ is some product of the elementary symmetric functions $e(i)$. We claim that in the sum over $S_k$, each term of the form
\[
\mathrm{sgn}(\sigma) h(g) f_{g_u}(\beta(\sigma(u))
\]
is canceled out by another unless 
\begin{equation} \label{g distinct}
g = (0,1,2,...,k-1).
\end{equation}
We prove the claim now. For a given $g$, take the smallest pair of indices $(i,j)$ in the lexicographic order such that $g_i = g_j$ and pair the same term arising from the term with $\sigma'$, where 
\[
\sigma'(i) = \sigma(j) \text{ and } \sigma'(j) = \sigma(i) \text{ and } \sigma'(m) = \sigma(m) \text{ otherwise.} 
\]
Thus the only terms that remain are those with $g_i$ all distinct. The only such $g$ is given by \eqref{g distinct} for which $h(g)=(-1)^k$. Therefore 
\begin{align*}
\det(R_1(\beta))^2 &= (\sum_{\sigma \in S_k} \mathrm{sgn}(\sigma) \prod_{u=1}^{k} f_{u-1}(\beta(\sigma(u))) )^2\\ 
&=\det(R_2(\beta))^2.
\end{align*}
Thus we have shown that 
\[
\Delta_k(E(n)) = \sum_{\beta\in \mathrm{Pairs}(k,n)} \mathcal{D}(\beta)^2 \det(R_2(\beta))^2.
\]
Using
\[
f_{u-1}(x_1,x_2) =\frac{x_1^u - x_2^u}{x_1 - x_2},  
\]
we write 
\begin{align*}
(\mathcal{D}(\beta)\det(R_2(\beta)))^2 &= (\sum_{\sigma \in S_k} \prod_{u=1}^k (z_{\beta(u,1)}^{\sigma(u)}-z_{\beta(u,2)}^{\sigma(u)}))^2\\ 
& = \det(R_3(\beta))^2.
\end{align*}
This completes the proof.
\end{proof}

\section{Proof of Equivalence of Minors}\label{equivalence}
We prove
 \begin{theorem}\label{minor equiv}
For integer $ k \geq 1$, we have
\[
\sum_{\beta \in \mathrm{Pairs}(k,n)} \det(R_3(\beta))^2 = n^{k-1} \sum_{b \in C(k+1,n)}D(b)^2.
\]

\end{theorem}
We first show how to index the terms on the left side. To do this, we define a set of functions $S(\beta)$. 
\begin{definition} 
For $\beta \in \mathrm{Pairs}(k,n)$, let $S(\beta)$ denote the set of $2^k$ functions $s$ such that the domain of each $s$ is the set of $k$ pairs $\beta(i), 1 \leq i \leq k$ and
such that the output $s(\beta(i))$ on the $i$-th pair of $\beta$ is 
\[
s(\beta(i)) \in \{ \beta(i,1), \beta(i,2)\}.
\]
That is, a function $s$ chooses one element in each pair of $\beta$.
For $s \in S(\beta)$, define 
\[
\mathrm{sgn}(s) = (-1)^N \text{ where } N = \#\{ i: s(\beta(i)) = \max(\beta(i,1), \beta(i,2))\}.
\]
For $s_1,s_2 \in S(\beta)$, define 
\[
\mathrm{sgn}(s_1,s_2) = (-1)^N \text{ where } N= \# \{ i: s_1(\beta(i)) \neq s_2 (\beta(i))\}. 
\]
Then 
\[
\mathrm{sgn}(s_1) \mathrm{sgn}(s_2) = \mathrm{sgn}(s_1,s_2).
\]
Define $z(\beta)$  
\[
z(\beta) = ((z_{\beta(1,0)}, z_{\beta(1,1)}), ..., (z_{\beta(k,0)}, z_{\beta(k,1)}) ).
\]
\end{definition}
 
 \begin{definition} 
Let $V =V(z_{j_1}, z_{j_2}, ..., z_{j_k}; \{1,2,...,k\})$ denote the $k \times k$ matrix with entries $V_{u,v}$ 
\[
V_{u,v} = z_{j_u}^v.
\]
Let $d(j_1,j_2,.., j_k)$ denote
\[
d(j_1,j_2,.., j_k) = \det(V).
\]
\end{definition}

\begin{remark} \label{well-defined}
\emph{Suppose for a given $\beta$ we have an ordering $\beta = \{\beta(1), ..., \beta(k)\}$. We denote $d(s(\beta))$ 
\[
d(s(\beta)) = d(s(\beta(1)), s(\beta(2)), ..., s(\beta(k))).
\] 
Note that the expression $d(s(\beta))$ depends up to sign on a choice of ordering on $\beta$. But for $s_1, s_2 \in S(\beta)$, the expression 
\[
d(s_1(\beta))d(s_2(\beta)) 
\] 
is independent of a choice of ordering, because for $\sigma$ in $S_k$ 
\begin{align*}
&d(s_1(\beta(\sigma(1))), ..., s_1(\beta(\sigma(k))))d(s_2(\beta(\sigma(1))), ..., s_2(\beta(\sigma(k))))\\ 
&=  \mathrm{sgn}(\sigma)^2 d(s_1(\beta(1)), ..., s(\beta_1(k))) d(s_2(\beta(1)), ..., s_2(\beta(k)))\\ 
&= d(s_1(\beta(1)),  ..., s(\beta_1(k))) d(s_2(\beta(1)) ..., s_2(\beta(k))). 
\end{align*}
}$\square$
\end{remark}

Applying the definition of $R_3(\beta)$, we get
\[
 \det(R_3(\beta))^2=  (\sum_{s \in S(\beta)} \mathrm{sgn}(s) d(s(\beta)))^2 .
\]
Then
\begin{align*}
\sum_{\beta \in \mathrm{Pairs}(k,n)} \det(R_3(\beta))^2 &= \sum_{\beta \in \mathrm{Pairs}(k,n)} (\sum_{s \in S(\beta)} \mathrm{sgn}(s) d(s(\beta)))^2 \\
&= \sum_{\beta \in \mathrm{Pairs}(k,n)} \sum_{s_1,s_2\in S(\beta)} \mathrm{sgn}(s_1,s_2)d(s_1(\beta)) d(s_2(\beta))
\end{align*}
The terms in the sum on the right are thus indexed by ordered triples 
\[
(\beta,s_1, s_2)
\]
for $\beta \in \mathrm{Pairs}(k,n)$ and $s_1,s_2 \in S(\beta)$. As discussed above, each term $d(s_1(\beta)) d(s_2(\beta))$ is well-defined independent of an ordering on $\beta$.

\begin{definition} 
Given such a triple $(\beta,s_1,s_2)$, choose an ordering on $\beta$ and let $I(\beta, s_1,s_2)$ be the set of indices
\[
I(\beta, s_1,s_2) = \{ i:s_1(\beta(i)) \neq s_2(\beta(i))\}.
\]
\end{definition}

\begin{theorem}\label{I<2}
\[
\sum_{\beta \in \mathrm{Pairs}(k,n)} \det(R_3(\beta))^2 = \sum_{ (\beta,s_1,s_2): |I(\beta,s_1,s_2)| \leq 1} \mathrm{sgn}(s_1,s_2)d(s_1(\beta))d(s_2(\beta)).
\]
\end{theorem}
\begin{proof}
Consider those terms with triples for which $|I(\beta,s(\beta_1,s_2)| \geq 2$.
We define a bijection to show that all such terms cancel.  Let $i_1,i_2 \in I(\beta, s_1,s_2)$ be the indices such that $
s_2(\beta(i_1)) \text{ and } s_2(\beta(i_2)) $
are the two smallest numbers in the set 
\[
\{ s_2(\beta(i)): i \in I(\beta, s_1,s_2)\}. 
\]
Note that the elements $\beta(i_1)$ and $\beta(i_2)$ of $\beta$ do not depend on the ordering on $\beta$.
Let $\beta' = \{\beta' \beta'(1), ..., \beta'(k)\}\in \mathrm{Pairs}(k,n)$ be defined by 
\[
\beta'(i) = 
\begin{cases}
\beta(i) &\text{ if } i \neq i_1,i_2\\
\{s_1(\beta(i_1)), s_2(\beta(i_2))\} &\text{ if } i =i_1\\
\{s_1(\beta(i_2)), s_2(\beta(i_1))\} & \text{ if } i =i_2.\\
\end{cases}
 \]
 Define $s_1', s_2' \in S(\beta')$ by 
 \[
 s_1'(\beta'(i)) = s_1(\beta(i))
 \]
 for all $1 \leq i \leq k$ and 
 \[
s_2'(\beta'(i)) = \begin{cases}
s_2(\beta(i)) &\text{ if } i \neq i_1,i_2\\
s_2(\beta(i_2)) &\text{ if } i =i_1\\
s_2(\beta(i_1)) & \text{ if } i =i_2.\\
\end{cases}
 \]
 This completes the definition of the bijection $(\beta,s_1,s_2) \mapsto (\beta',s_1',s_2')$. We next show that 
 \[
 \mathrm{sgn}(s_1',s_2')d(s_1'(\beta'))d(s_2'(\beta'))=-\mathrm{sgn}(s_1,s_2)d(s_1(\beta))d(s_2(\beta)).
 \]
 By construction 
 \[
 I(\beta,s_1,s_2) = I(\beta',s_1',s_2') 
 \]
 so \[
  \mathrm{sgn}(s_1',s_2') =  \mathrm{sgn}(s_1,s_2).
  \]
Without loss of generality assume $i_1 < i_2$. Then
\begin{align*}
&d(s_1'(\beta'))d(s_2'(\beta')) \\ 
&= d(s_1'(\beta'(1)), ..., s_1'(\beta'(k)))d(s_2'(\beta'(1)), ..., s_2'(\beta'(k)))  \\ 
&=   d(s_1(\beta(1)), ..., s_1(\beta(k)))\\ 
&\times  d(s_2(\beta(1)), ..., s_2(\beta(i_1-1)), s_2(\beta(i_2)),s_2(\beta(i_1+1)), ..., s_2(\beta(i_2-1)), s_2(\beta(i_1)),s_2(\beta(i_2+1)),...,s_2(\beta(k))) \\ 
&=   -d(s_1(\beta(1)), ..., s_1(\beta(k)))\\ 
&\times  d(s_2(\beta(1)), ..., s_2(\beta(i_1-1)), s_2(\beta(i_1)),s_2(\beta(i_1+1)), ..., s_2(\beta(i_2-1)), s_2(\beta(i_2)),s_2(\beta(i_2+1)),...,s_2(\beta(k))) \\ 
& = -d(s_1(\beta))d(s_2(\beta)).
\end{align*}

\end{proof}
We next show how an element $\beta \in \mathrm{Pairs}(k,n)$ corresponds to a graph $G(\beta)$, and how $s \in S(\beta)$ directs the edges of $G(\beta)$.
\begin{definition}
For $\beta \in \mathrm{Pairs}(k,n)$ we define a graph $G(\beta)$. The set of vertices $V(G(\beta))$ of $G(\beta)$ is the set of distinct numbers that appear as a $\beta(i,1)$ or $\beta(i,2)$. Thus  
\[
|V(G(\beta))|  = |\beta|
\]
The edge set $E(G(\beta))$ is 
\[
E(G(\beta)) = \{ (\beta(i,1), \beta(i,2)): 1 \leq i \leq k\}.
\]
Given an $s \in S(\beta)$, for each $i$ write
\[
\beta(i) = \{ \beta_i, s(\beta_i)\} .
\]
Then we say that the edge $(s(\beta_i), \beta_i)$ is an outgoing edge from $s(\beta_i)$ and an incoming edge to $\beta_i$. Denote the resulting directed graph by $G(\beta,s)$. 
For $0\leq h\leq k-1$, let $\mathcal{G}(k,h)$ denote the set of graphs $G(\beta)$ such that 
\[
|V(G(\beta))| = k+h+1.
\] 
\end{definition} 

\begin{lemma}\label{no cycles}
For integer $k \geq 1$, 
\[
 \sum_{\beta \in \mathrm{Pairs}(k,n)} \det(R_3(\beta))^2 = \sum_{ (\beta,s_1,s_2)} \mathrm{sgn}(s_1,s_2)d(s_1(\beta))d(s_2(\beta))
 \]
 where the sum on the right is over all triples $(\beta,s_1,s_2)$ such that $\beta$ has no repeated pairs and: 
 
 1. $|I(\beta,s_1,s_2)| \leq 1$. 
 
 2. $G(\beta,s_1)$ and $G(\beta,s_2)$ each have no vertex with more than one outgoing edge.
 
 3. $G(\beta)$ has no cycles. 
 
\end{lemma}

\begin{proof}
If $\beta$ has a repeated pair,that $ \beta(i) = \beta(j)$ for some $i \neq j$, then
\[
\det(R_3(\beta)).
\] 
Thus we may assume $\beta$ has no repeated pairs.

Statement 1 was proven above. Statement 2 follows from the fact that if $G(\beta,s)$ has a vertex with more than one outgoing edge, then the determinant $d(s(\beta))$ has a repeated index and thus equals 0. To prove statement 3, suppose that $G(\beta)$ has a cycle. By statement 2, for $d(s(\beta))$ to be non-zero, $s$ must make each cycle in $G(\beta)$ a directed cycle. Thus we take the cycle $C$ of $G(\beta)$ whose vertex set is smallest in the lexicographic ordering and write the edges as 
\[
(v_1,v_2), (v_2,v_3), ...., (v_{m-1},v_m), (v_m, v_1)
\]  
for some $m \geq 2$. Without loss of generality we may assume $s_1((v_i, v_{i+1})) =v_i$ and $s_1((v_m,v_1)) =v_m$. We then match the triple $(\beta,s_1,s_2)$ to $(\beta,s_1', s_2)$ where $s_1'$ is the same as $s_1$ except that it reverses the cycle $C$. Thus 
\[
\mathrm{sgn}(s_1) = (-1)^m \mathrm{sgn}(s_1')
\]
Then in $s_1(\beta)$ we have the subsequence
\[
\{ v_1, v_2, ..., v_m\}
\]
and in $s_1'(\beta)$ we have the subsequence 
\[
\{ v_2, v_3, ..., v_m, v_1\}.
\]
Thus 
\[
d(s_1(\beta)) = (-1)^{m-1} d(s_1'(\beta)).
\]
Thus the terms from $I(\beta, s_1, s_2)$ and $I(\beta, s_1', s_2)$ cancel. This proves the statement 3.  
\end{proof}

Now we index the terms on the right side of Theorem \ref{minor equiv}. 
\begin{lemma}
For $b \in C(k+1,n)$, 
\[
D(b)^2 = \left (\sum_{i=1}^{k+1} (-1)^{k+1-i} d(\hat{b}_i) \right)^2
\] 
\end{lemma}
\begin{proof}
This follows from Lemma \ref{identity} or by taking the definition of $D(b)$ as the determinant of a Vandermonde matrix and expanding along the row with all 1's. 
\end{proof}
Thus 
\begin{equation} \label{D(b) sum}
\sum_{b \in  C(k+1,n)} D(b)^2 = \sum_{b \in C(k+1,n)} \left(\sum_{i=1}^{k+1} d(\hat{b}_i)^2 + \sum_{j=1}^{k+1}\sum_{i=1, \neq j}^{k+1} (-1)^{i+j}d(\hat{b}_i)  d(\hat{b}_j) \right)
\end{equation}
We consider the sum 
\[
\sum_{b \in C(k+1,n)}  \sum_{j=1}^{k+1}\sum_{i=1, \neq j}^{k+1} (-1)^{i+j}d(\hat{b}_i)  d(\hat{b}_j). 
\]
This sum is indexed by ordered triples
\[
(b,i,j)
\]
for $1 \leq i \neq j \leq k+1$. 
In the sum from \eqref{I<2}, we map these terms to sets of terms in the sum
\[
 \sum_{ (\beta,s_1,s_2): |I(\beta,s_1,s_2)| = 1} \mathrm{sgn}(s_1,s_2)d(s_1(\beta))d(s_2(\beta)).
\]
We thus associate to each such triple $(b,i,j)$ a set of triples $(\beta, s_1, s_2)$:

 \begin{definition}
 Let $b \in C(k+1,n)$ and $1 \leq i\neq j \leq k+1$.
 Given a triple $(b,i,j)$, define $\mathcal{G}(b,i,j)$ to be the set of all $(\beta, s_1, s_2)$ such that 
\[
s_1(\beta) \text{ as a set is equal to }  \hat{b}_i
\] 
and 
\[
s_2(\beta) \text{ as a set is equal to }  \hat{b}_j
\]
and 
\[
d(s_1(\beta))d(s_2(\beta))\neq 0.
\]
Define the subset $\mathcal{G}(b,i,j;h)\subset \mathcal{G}(b,i,j)$ to consist of all triples $(\beta, s_1, s_2)$
such that 
\[
|\beta|= k+1+h.
\]
Thus 
\[
\mathcal{G}(b,i,j;h) = \mathcal{G}(b,i,j) \cap \mathcal{G}(k,h).
\]

Given a pair $(b,i)$, define $\mathcal{G}(b,i)$ to be the set of all $(\beta, s)$ with $s \in S(\beta)$ such that as sets
\[
s(\beta) = \hat{b}_i \text{ and } d(s(\beta)) \neq 0.
\]
Define the subset $\mathcal{G}(b,i;h)\subset \mathcal{G}(b,i)$ to consist of those pairs $(\beta,s)$ such that 
\[
|\beta|= k+1+h.
\]
\end{definition}

We characterize the possible $(\beta,s_1,s_2)$ and $(\beta,s)$that can appear in $\mathcal{G}(b,i,j;h)$ and $\mathcal{G}(b,i;h)$, respectively. We identify $\beta$ with the graph $G(\beta)$ in the following lemma. 

\begin{lemma} \label{Gij}
Let $b \in C(k+1,n)$. 
Then $\mathcal{G}(b,i,j;h)$ is the set of all $(\beta,s_1,s_2)$ such that for the graphs $G(\beta)$: 

1.  $b_1, ..., b_{k+1}$ are vertices of $G(\beta)$, and $|V(G(\beta))|=h+k+1$.

2.  There are exactly $k$ edges. 

3. There are exactly $h+1$ components. The vertices $b_i$ and $b_j$ are in the same component, and there is exactly one non-$b$ vertex in each of the remaining $h$ components.
Here, a ``non-$b$" vertex means a number not in $b$.  

4. Each component has at least two vertices. 

5. There are no cycles, loops, or multiple edges.

6. $s_1$ and $s_2$ direct each component of $G(\beta)$ such that the unlabeled vertex corresponds to the root vertex of a directed tree, and on the component with $b_i$ and $b_j$, $s_1$ makes $b_i$ the root and $s_2$ makes $b_j$ the root.

\end{lemma}
\begin{proof}
Statement 1 follows from the definition of $h$ and the fact that as sets
\[
s_1(\beta)\cup s_2(\beta) = b.
\]
Statement 2 follows from the requirement that $\beta \in \mathrm{Pairs}(k,n)$. That there are no cycles follows from Lemma \ref{no cycles}, and there are no loops or multiple edges by construction. Statement 4 also follows by construction. The fact that an $s$ must pick a vertex from each edge and that an $s$ cannot pick the same vertex from two edges (or else $d(s(\beta))=0$ by Lemma \ref{no cycles}) means that $s$ makes any component into a rooted tree, such that the root is not in $d(s(\beta))$. Thus a component can have at most one non-$b$ vertex, and $s_1$ and $s_2$ must agree on all components that have a non-$b$ vertex. Thus $b_i$ and $b_j$ must be the same component, and $s_1$ makes $b_i$ the root of this component and $s_2$ makes $b_j$ the root. This proves statement 3 and 6.
\end{proof}

\begin{definition}
For integer $j \geq 0$ let $(x)_j$ denote the falling factorial 
\[
(x)_j = \prod_{i=1}^j(x-i+1). 
\]
\end{definition}

\begin{definition}
Let $b = \{1,2,...,k+1\}$ and take a graph $G$ in 
\[
\mathcal{G}(b,1,2;h).
\] 
Now make a graph $G'$ by taking every non-$b$ vertex in $G$ an unlabeled vertex. The set of such $G'$ depends only on $k$ and $h$; let $\mathcal{A}_1(k,h)$ denote this set.
\end{definition}

\begin{lemma}\label{counting Gbijh}
For $0 \leq h \leq k-1$ and any $b \in C(k+1, n)$,
\[
|\mathcal{G}(b,i,j;h)| = |\mathcal{A}_1(k,h)| (n-k-1)_h.
\]
\end{lemma}
\begin{proof} 
We can re-name the elements of $b$ to correspond to $\{1,2,...,k+1 \}$
with $b_i$ corresponding to 1 and $b_j$ corresponding to 2. Then any graph in $\mathcal{G}(b,i,j;h)$ corresponds to taking a graph in $\mathcal{A}_1(k,h)$ and labeling the unlabeled vertices with numbers chosen from the set $\{ 1,2,...,n\} - b$. There are 
\[
(n-k-1)(n-k-2)...(n-k-h)
\]
ways to do this because each unlabeled vertex appears in a component with a vertex in $b$ and we may order these components. 
\end{proof}

\begin{lemma}

\begin{align*}
&\sum_{ \beta \in \mathrm{Pairs}(k,n), |I(\beta,s_1,s_2)| = 1} \mathrm{sgn}(s_1,s_2)d(s_1(\beta))d(s_2(\beta)) \\
&=\left(\sum_{h=0}^{k-1} |\mathcal{A}_1(k,h)|(n-k-1)_h\right) \sum_{b \in C(k+1,n)}  \sum_{j=1}^{k+1}\sum_{i=1, \neq j}^{k+1} (-1)^{i+j}d(\hat{b}_i)  d(\hat{b}_j)
\end{align*}
\end{lemma}
\begin{proof}
Each $(\beta, s_1,s_2)$ with $|I(\beta,s_1,s_2)| = 1$ is in a unique $\mathcal{G}(b,i,j)$ with $i \neq j$: the sets $s_1(\beta)$ and $s_2(\beta)$ determine $b$, and $b_i$ is the only element of $b$ not in $s_1(\beta)$ and $b_j$ is the only element of $b$ not in $s_2(\beta)$. Thus 
\begin{align*}
 &\sum_{ \beta \in \mathrm{Pairs}(k,n), |I(\beta,s_1,s_2)| = 1} |\mathrm{sgn}(s_1,s_2)d(s_1(\beta))d(s_2(\beta))|\\
=& \sum_{b \in C(k+1,n)}  \sum_{j=1}^{k+1}\sum_{i=1, \neq j}^{k+1} |G(b,i,j)| |(-1)^{i+j}d(\hat{b}_i)  d(\hat{b}_j)|  \\ 
=&  \left(\sum_{h=0}^{k-1} |\mathcal{A}_1(k,h)|(n-k-1)_h\right) \sum_{b \in C(k+1,n)}  \sum_{j=1}^{k+1}\sum_{i=1, \neq j}^{k+1}  |(-1)^{i+j}d(\hat{b}_i)  d(\hat{b}_j)|
\end{align*}
by Lemma \ref{counting Gbijh}. 
All we have to check now is that the signs agree. That is, we show that for $(\beta,s_1,s_2) \in \mathcal{G}(b,i,j)$
\[
\mathrm{sgn}(s_1,s_2)d(s_1(\beta))d(s_2(\beta)) = (-1)^{i+j}d(\hat{b}_i)  d(\hat{b}_j).
\]
 Since $G(\beta)$ has no cycles and vertices $b_i$ and $b_j$ are in the same component $C_0$ there is a unique path from $b_i$ to $b_j$ whose edges we write as
\[
(b_i, v_1), (v_1, v_2), ..., (v_{m-1},v_m), (v_m, b_j).
\]
Then $s_1$ and $s_2$ agree on all other pairs of $\beta$ and 
\[
\mathrm{sgn}(s_1,s_2) = (-1)^{m+1}.
\]
Order $\beta$ such that 
\[
\beta(1) = \{b_1, v_1\},  \, \,\, \beta(m+1) = \{v_m, b_j\}, \,\,\,\beta(i+1) = \{v_i, v_{i+1}\} \text{ for } 1 \leq i \leq m.
\]
Thus 
\begin{align*}
d(s_1(\beta)) &= d(v_1, v_2, ..., v_m, b_j,\{ l\}),\\ 
d(s_2(\beta)) &= d(b_i,v_1, v_2, ..., v_m, \{l\})\\ 
& = (-1)^{m}d(v_1, v_2, ..., v_m,b_i, \{ l\}).
\end{align*}
where $l$ indicates some sequence.
Without loss of generality assume $i <j$. Then
\begin{align*}
d(\hat{b}_i) &= d(b_1, ..., b_{i-1}, b_{i+1}, ..., b_{k+1})\\ 
d(\hat{b}_j) &= d(b_1, ..., b_{j-1}, b_{j+1}, ..., b_{k+1})\\ 
&= (-1)^{j-i-1}d(b_1, ..., b_{j-1}, b_i, b_{j+1}, ..., b_{k+1}).
\end{align*}
Therefore we may choose some permuations $\sigma, \tau \in S_{k}$ such that
\begin{align*}
\mathrm{sgn}(s_1,s_2)d(s_1(\beta))d(s_2(\beta)) &= (-1)^{m+1}(-1)^m \mathrm{sgn}(\sigma)^2 d(b_i, \hat{b}_{i,j}) d(b_j, \hat{b}_{i,j})\\ 
\end{align*}
and 
\[
(-1)^{i+j}d(\hat{b}_i)  d(\hat{b}_j) = (-1)^{i+j} (-1)^{j-i-1} \mathrm{sgn}(\tau)^2d(b_i, \hat{b}_{i,j}) d(b_j, \hat{b}_{i,j})
\]
where 
\[
\hat{b}_{i,j} = (b_1, ..., b_{i-1}, b_{i+1}, ..., b_{j-1}, b_{j+1}, ..., b_{k+1}).
\]
This completes the proof.
\end{proof}

Next we prove that 
\begin{equation}\label{Ahk}
\sum_{h=0}^{k-1} |\mathcal{A}_1(k,h)|(n-k-1)_h = n^{k-1}.
\end{equation}
We prove this by counting the number of graphs described above. The result is 
\begin{theorem}
The number of forests such that there is exactly one non-rooted tree that contains the vertices 1 and 2; the remaining trees are rooted; there are exactly $k+1$ non-root vertices;  the non-root vertices chosen from the set $\{1,2,...,k+1 \}$; all non-root vertices are chosen the set $\{ k+2, k+3, ..., n\}$; and every tree has at least two vertices is 
\[
n^{k-1}.
\] 
\end{theorem}

We first determine the relations among the coefficients $|\mathcal{A}_1(k,h)|$ implied by \eqref{Ahk}.
\begin{lemma}\label{Amj}
For integer $k\geq0$,
\[
x^{k-1} = \sum_{h=0}^{k-1} B(k,h)(x-k-1)_h.
\]
where the numbers $B(k,h)$ are determined by the following relations:
\[
B(k,h) = B(k-1,h-1) + (1+2h+k)B(k-1,h)+ (1+h)(1+h+k)B(k-1,h+1)
\]
\[
B(k,-1)=0,\,\,\,\,\, B(0,h)=  \delta_{0,h}.
\]
\end{lemma}
\begin{proof}
We use induction on $k$. The lemma is true for $k=0$. Assume it is true for $k \geq 0$. Then take
\[
x^{k} =\sum_{h=0}^{k-1} B(k,h)x(x-k-1)_h
\]
and re-express the right side in the basis 
\[
(x-(k+1)-1)_h
\]
for $0 \leq h \leq k$. Computing the coefficients in this basis in terms of $B(k,h)$ completes the proof.
\end{proof}

\begin{theorem}For $0 \leq h \leq k-1$,
\[
\mathcal{A}_1(k,h)|=B(k,h).
\]  
 \end{theorem}
 \begin{proof}
We construct $\mathcal{A}_1(k,h)$ from the three sets  $\mathcal{A}_1(k-1,h-1), \mathcal{A}_1(k,h)$, and $\mathcal{A}_1(k-1,h+1)$ in the following six steps.
For a graph $G$, we let $(u,v)$ denote an undirected edge in $G$ between the vertices $u$ and $v$. 

\vspace{1cm}
1. For $G \in \mathcal{A}_1(k-1,h-1)$, we adjoin to $G$ the component that consists of the vertex $(k+1)$ with one edge to an unlabeled vertex. This contributes 
\[
|\mathcal{A}_1(k-1,h-1)|
\]
graphs to $\mathcal{A}_1(k,h)$. 

\vspace{1cm}
2. For $G \in \mathcal{A}_1(k-1,h)$, for each vertex $v$ in $G$ whether labeled or unlabeled, we create a graph $G'$ by adjoining one edge
\[
(k+1, v).
\]
This contributes 
\[
(k+h)|\mathcal{A}_1(k-1,h)|
\]
graphs to $\mathcal{A}_1(k,h)$, as each $G \in \mathcal{A}_1(k-1,h)$ has $k+h$ vertices.

\vspace{1cm}
3. For $G \in \mathcal{A}_1(k-1,h)$, for each unlabeled vertex $v$ in $G$, we create a graph $G'$ by labeling $v$ as $(k+1)$ and then adjoining the edge 
\[
(k+1,v).
\]
This contributes 
\[
h|\mathcal{A}_1(k-1,h)|
\]
graphs to $\mathcal{A}_1(k,h)$, as each $G \in \mathcal{A}_1(k-1,h)$ has $h$ unlabeled vertices.

\vspace{1cm}
4. For $G \in \mathcal{A}_1(k-1,h)$, we take the component $C_0$ that contains the vertices 1 and 2 and let $e$ denote the edge 

\[
e=(2,v)
\]
such that in $G\backslash e$ the vertices 1 and 2 are in separate components. Then create the graph $G'$ by adjoining the edges 
\[
(k+1,2) \text{ and } (k+1,v).
\]
This contributes
\[
|\mathcal{A}_1(k-1,h)|
\]
graphs to $\mathcal{A}_1(k,h)$. 

Thus in total the graphs in $\mathcal{A}_1(k-1,h)$ contribute 
\[
(1+2h+k)|\mathcal{A}_1(k-1,h)|
\]
graphs to $\mathcal{A}_1(k,h)$.

\vspace{1cm}
5. For $G \in \mathcal{A}_1(k-1,h+1)$, let the components be
\[
C_0, C_1, ..., C_{h+2}
\]
where $C_0$ is the component containing the vertices 1 and 2. Take the unlabeled vertex $v_i$ in the component $C_i$ (so $i \neq 0$).  Label $v$ as $(k+1)$, and then for each vertex $v'$ not in $C_i$, we create the graph $G'$ by adjoining the edge 
\[
(v,v').
\]
This contributes
\[
k+h+1 - |V(C_i)|
\]
graphs to $\mathcal{A}_1(k,h)$. Doing this for each unlabeled $v_i$ in $G$ contributes 
\begin{align} \nonumber
(k+h+1)(h+1) - \sum_{i=1}^{h+1} |V(C_i)| &= (k+h+1)(h+1) -(k+h+1) +|V(C_0)|  \\ \label{vc0}
&= (k+h+1)h+  |V(C_0)|
\end{align}
graphs to $\mathcal{A}_1(k,h)$.

\vspace{1cm}
6. For the same $G$ as in step 5, we take the component $C_0$ and let $e$ denote the edge 
\[
e=(2,v)
\]
such that in $G\backslash e$ the vertices 1 and 2 are in separate components $U_1$ and $U_2$, respectively. Then, in $G$, for each $v'$ in a component $C_i$, $i \neq 0$, we label the unlabeled vertex in $C_i$ as $(k+1)$, remove the edge $e$, and adjoin the edges 
\[
(v,k+1) \text{ and } (2,v'). 
\]
This contributes 
\begin{equation}\label{vci}
\sum_{i=1}^{h+1} |V(C_i)|
\end{equation}
graphs to $\mathcal{A}_1(k,h)$.
Adding \eqref{vc0} and \eqref{vci} yields
\[
 (k+h+1)h+  |V(C_0)|+ \sum_{i=1}^{h+1} |V(C_i)| =  (k+h+1)(h+1)
\] 
graphs in $\mathcal{A}_1(k,h)$ that come from one $G$ in $\mathcal{A}_1(k-1,h+1)$. Doing this for each $G \in \mathcal{A}_1(k-1,h+1)$ contributes 
\[
 (k+h+1)(h+1)|\mathcal{A}_1(k-1,h+1)|
\]
graphs in $\mathcal{A}_1(k,h)$.

This accounts for every graph in  $\mathcal{A}_1(k,h)$: take the vertex $ v=k+1$ in $\mathcal{A}_1(k,h)$ and exactly one of the following is true: 

1. A component consists solely of $v$ and an unlabeled vertex. 

2. $v$ is a leaf in a component that contains at least three vertices. 

3. $v$ is adjacent to an unlabeled vertex which is a leaf in a component that contains at least three vertices. 

4. $v$ is in component $C_0$, and $v$ is adjacent to exactly two vertices, one of which is 2, and $G\backslash v$ separates vertices 1 and 2. 

5. $v$ is in component $C_0$, and $v$ is not a leaf, and $G\backslash v$ does not separate vertices 1 and 2. 

6. $v$ is in component $C_0$, $G\backslash v$ separates vertices 1 and 2, and either $v$ is adjacent to more than two vertices, or $v$ is adjacent to exactly two vertices, neither of which is 2.

This completes the proof.
\end{proof}

We now consider the sum from \eqref{D(b) sum}
\[
\sum_{b \in C(k+1,n)} \sum_{i=1}^{k+1} d(\hat{b}_i)^2.
\]
This is equal to 
\[
(n-k)\sum_{b \in C(k,n)} d(b)^2.
\]
And 
\begin{align*}
\sum_{ (\beta,s_1,s_2): |I(\beta,s_1,s_2)| = 0} \mathrm{sgn}(s_1,s_2)d(s_1(\beta))d(s_2(\beta)) = \sum_{ (\beta,s)} d(s(\beta))^2
\end{align*}

We thus show that 
\begin{align} \label{I=0}
\sum_{ (\beta,s)} d(s(\beta))^2 = n^{k-1}(n-k)\sum_{b \in C(k,n)}d(b)^2.
\end{align}
 
 \begin{definition}
 Let $b \in C(k,n)$.
 Define $\mathcal{G}(b,i,j)$ to be the set of all $(\beta, s) \in B_0^*(k,n)$ such that 
\[
s(\beta) \text{ as a set is equal to }  b.
\] 

Define the subset $\mathcal{G}(b;h)\subset \mathcal{G}(b)$ to consist of all $(\beta, s)$
such that 
\[
|\beta|= k+1+h.
\]
\end{definition}

We characterize the possible $(\beta,s)$that can appear in $\mathcal{G}(b;h)$. We identify $\beta$ with the graph $G(\beta)$ in the following lemma. 

\begin{lemma} \label{Gb}
Let $b \in C(k,n)$. 
Then $\mathcal{G}(b;h)$ is the set of all $(\beta,s_)$ such that for the graphs $G(\beta)$: 

1.  $b_1, ..., b_k$ are vertices of $G(\beta)$, and $|V(G(\beta))|=h+k+1$.

2. There are exactly $k$ edges. 

3. There are exactly $h+1$ components. There is exactly one non-$b$ vertex in each component.
Here, a ``non-$b$" vertex means a number not in $b$.  

4. Each component has at least two vertices. 

5. There are no cycles, loops, or multiple edges.

6. $s$ makes each component into a directed rooted tree such that the unlabeled vertex is the root.
\end{lemma}
\begin{proof}
This follows from the same reasoning in Lemma \ref{Gij}.
\end{proof}

\begin{definition}
Let $b = \{1,2,...,k\}$ and take a graph $G$ in 
\[
\mathcal{G}(b;h).
\] 
Now make a graph $G'$ by taking every non-$b$ vertex in $G$ an unlabeled vertex. The set of such $G'$ depends only on $k$ and $h$; let $\mathcal{A}_0(k,h)$ denote this set.
\end{definition}

\begin{lemma}\label{counting Gb}
For $0 \leq h \leq k-1$ and any $b \in C(k, n)$,
\[
|\mathcal{G}(b;h)| = |\mathcal{A}_0(k,h)| (n-k)_{h+1}.
\]
\end{lemma}
\begin{proof} 
We can re-name the elements of $b$ to correspond to $\{1,2,...,k \}$. Then any graph in $\mathcal{G}(b;h)$ corresponds to taking a graph in $\mathcal{A}_0(k,h)$ and labeling the unlabeled vertices with numbers chosen from the set $\{ 1,2,...,n\} - b$. There are 
\[
(n-k)(n-k-1)...(n-k-h)
\]
ways to do this because each unlabeled vertex appears in a component with a vertex in $b$ and we may order these components. 
\end{proof}

\begin{lemma}

\begin{align*}
&\sum_{ \beta \in \mathrm{Pairs}(k,n), |I(\beta,s_1,s_2)| = 0} d(s(\beta))^2 \\
&=\left(\sum_{h=0}^{k-1} |\mathcal{A}_0(k,h)|(n-k)_{h+1}\right) \sum_{b \in C(k,n)}  d(b)^2
\end{align*}
\end{lemma}
\begin{proof}
Each $(\beta, s)$ is in a unique $\mathcal{G}(b)$: the set $s_1(\beta)$ is $b$. Thus 
\begin{align*}
 &\sum_{ \beta \in \mathrm{Pairs}(k,n), |I(\beta,s_1,s_2)| = 0} d(s(\beta))^2\\
=& \sum_{b \in C(k,n)} |\mathcal{G}(b)| d(b)^2  \\ 
=&  \left(\sum_{h=0}^{k-1} |\mathcal{A}_0(k,h)|(n-k)_{h+1}\right) \sum_{b \in C(k,n)}  d(b)^2
\end{align*}
by Lemma \ref{counting Gb}. 
This completes the proof.
\end{proof}

We thus show that 
\[ 
\sum_{h=0}^{k-1} |\mathcal{A}_0(k,h)|(n-k)_{h+1}  = n^{k-1}(n-k).
\]
The result is 
\begin{theorem}
The number of rooted forests such that there are exactly $k$ non-root vertices; the non-root vertices chosen from the set $\{1,2,...,k \}$; the root vertices are chosen from the set $\{k+1, k+2, ..., n \}$; and every tree has at least two vertices is 
\[
n^{k-1}(n-k).
\] 
\end{theorem}
Now
\begin{align}\label{Ahk}
\sum_{h=0}^{k-1} |\mathcal{A}_0(k,h)|(n-k)_{h+1} = (n-k)\sum_{h=0}^{k-1} |\mathcal{A}_0(k,h)|(n-k-1)_h
\end{align}
We thus show that 
\[
\sum_{h=0}^{k-1} |\mathcal{A}_0(k,h)|(n-k-1)_h = n^{k-1}.
\]
That is, we prove
\[
|\mathcal{A}_0(k,h)| = B(k,h)
\]

\begin{theorem}For $0 \leq h \leq k-1$,
\[
A_0(k,h) = B(k,h).
\]  
 \end{theorem}

\begin{proof}
We construct $\mathcal{A}_0(k,h)$ from the three sets  $\mathcal{A}_0(k-1,h-1), \mathcal{A}_0(k,h)$, and $\mathcal{A}_0(k-1,h+1)$ in the following six steps.
For a graph $G$, we let $(u,v)$ denote an undirected edge in $G$ between the vertices $u$ and $v$. 

\vspace{1cm}
1. For $G \in \mathcal{A}_0(k-1,h-1)$, we adjoin to $G$ the component that consists of the vertex $k$ with one edge to an unlabeled vertex. This contributes 
\[
A_0(k-1,h-1)
\]
graphs to $\mathcal{A}_0(k,h)$. 

\vspace{1cm}
2. For $G \in \mathcal{A}_0(k-1,h)$, for each vertex $v$ in $G$ whether labeled or unlabeled, we create a graph $G'$ by adjoining the edge
\[
(k, v).
\]
This contributes 
\[
(k+h)A_0(k-1,h)
\]
graphs to $\mathcal{A}_0(k,h)$, as each $G \in \mathcal{A}(k-1,h)$ has $k+h$ vertices.

\vspace{1cm}
3. For $G \in \mathcal{A}_0(k-1,h)$, for each unlabeled vertex $v$ in $G$, we create a graph $G'$ by labeling $v$ as $k$ and then adjoining the edge 
\[
(k,v).
\]
This contributes 
\[
(h+1)A_0(k-1,h)
\]
graphs to $\mathcal{A}_0(k,h)$, as each $G \in \mathcal{A}_0(k-1,h)$ has $h+1$ unlabeled vertices. Thus in total the graphs in $\mathcal{A}_0(k-1,h)$ contribute 
\[
(1+2h+k)A_0(k-1,h)
\]
graphs to $\mathcal{A}_0(k,h)$.

\vspace{1cm}
4. For $G \in \mathcal{A}_0(k-1,h+1)$, let the components be
\[
C_1, C_1, ..., C_{h+2}.
\]
Take the unlabeled vertex $v_i$ in the component $C_i$.  Label $v$ as $k$, and then for each vertex $v'$ not in $C_i$, we create the graph $G'$ by adjoining the edge 
\[
(v,v').
\]
This contributes
\[
k+h+1 - |V(C_i)|
\]
graphs to $\mathcal{A}_0(k,h)$. Doing this for each $i$ contributes 
\begin{align} \nonumber
(k+h+1)(h+2) - \sum_{i=1}^{h+2} |V(C_i)| &= (k+h+1)(h+2) -(k+h+1) \\
&= (k+h+1)(h+1)
\end{align}
graphs to $\mathcal{A}_0(k,h)$.

This accounts for every graph in  $\mathcal{A}_0(k,h)$: take the vertex $ v=k$ in $\mathcal{A}_0(k,h)$ and exactly one of the following is true: 

1. A component consists solely of $v$ and an unlabeled vertex. 

2. $v$ is a leaf in a component that contains at least three vertices. 

3. $v$ is adjacent to an unlabeled vertex which is a leaf in a component that contains at least three vertices. 

4. None of the above.

This completes the proof.
\end{proof}

\section{Further Work}

\begin{itemize} 

\item Analyze the coefficients of $e_i e_j$ in the entries of $E(n)$. 

\item Prove the equivalence of minors using the Newton-Girard identities to express the power-sum functions in terms of the elementary symmetric functions, instead of using the indeterminates $z_i$. 

\item See if there is some family relating the matrix $E(n)$ and the Bezoutian matrix, or try to characterize all matrices that have equivalent minors.

\item See if tensors can be applied instead of just matrices.

\item See if these expressions for the inequalities can be applied to prove the convergence of the NRS($m$) algorithms of \cite{DefrancoNRS}. 

\item Apply these expressions for the inequalities to the Jensen polynomials of the Riemann xi function, using the integral kernels in 
\cite{Csordas}, \cite{DefrancoXi}, or the kernel used in Li's criterion.

\item See if the these expressions can be generalized to other root systems. 

\item Use these expressions to directly prove that they determine when a polynomial has real zeros.


\end{itemize}

\end{document}